\documentclass[12pt,a4paper,reqno]{article}
\usepackage[left=2cm,right=2cm,top=1.8cm,bottom=2cm]{geometry}
\usepackage{lineno}

\usepackage[utf8]{inputenc}
\usepackage[english]{babel}
\usepackage{amsmath}
\usepackage{amsfonts}
\usepackage{amssymb}
\usepackage[titletoc,title]{appendix}
\usepackage{color}
\usepackage{mathtools}
\usepackage{bbm}
\usepackage{bm}
\usepackage{enumerate}
\usepackage{mathrsfs}
\usepackage{esint}
\usepackage{dsfont}
\usepackage{amsthm}
\usepackage{cite}
\usepackage{url}
\usepackage{cases}
\usepackage{hyperref}
\usepackage{cleveref}
\hypersetup{%
        colorlinks=true,
        citecolor={black},
        linkcolor={black},
        urlcolor={black},}

\newcommand*{\NN}{\mathbb{N}}

\newcommand*{\RR}{\mathbb{R}}

\newcommand*{\LL}{\mathcal{L}}

\newcommand*{\xrarr}[2][]{\xrightarrow[#2]{#1}}

\newcommand*{\sset}{\subseteq}
\newcommand*{\smin}{\setminus}
\newcommand*{\deq}{:=}
\newcommand*{\e}{\varepsilon}
\newcommand*{\g}{\gamma}

\renewcommand*{\d}{\mathrm{d}}
\newcommand*{\1}{\chi}
\newcommand*{\pd}{\partial}

\theoremstyle{plain}
\newtheorem{thm}{Theorem}
\newtheorem{proposition}[thm]{Proposition}
\newtheorem{lemma}[thm]{Lemma}

\theoremstyle{definition}

\theoremstyle{remark}
\newtheorem{rmk}[thm]{Remark}

\renewcommand{\l}{\lambda}
\renewcommand{\O}{\Omega}

\newcommand*{\bra}[1]{\left\lbrace #1 \right\rbrace}

\makeatletter
\def\blfootnote{\gdef\@thefnmark{}\@footnotetext}
\makeatother

\title{A localized boundary deformation\\which splits the spectrum of the Laplacian}

\date{June 12, 2017}

\author{Alexander Dabrowski\footnote{Department of Mathematics, ETH Z\"urich, Switzerland.
}}

\begin{document}
  
\maketitle
\blfootnote{\textit{Mathematics subject classification}: 35J25, 35P15, 58C40.}
 
\begin{abstract}
For any Lipschitz domain we construct an arbitrarily small, localized perturbation which splits the spectrum of the Laplacian into simple eigenvalues.
We use for this purpose a Hadamard's formula and spectral stability results.
\end{abstract}

\section{Introduction}

In the seminal works \cite{MichelettiLaplace} and  \cite{Uhlenbeck}, respectively Micheletti and Uhlenbeck showed that the eigenvalues of the Dirichlet Laplacian are generically simple in the space of smooth manifolds equipped with the $C^k$-topology (see also the survey papers \cite[Section 4.3]{BurenkovLambertiCristoforis}, \cite[Section 1.3]{Hale} and references therein for related works).
In this paper we generalize this result to Lipschitz domains  and show that a stronger, localized version holds as follows.

\begin{thm}
\label{T:EllipticLocalSpectrumGen}
For any Lipschitz domain $\O$, $\e > 0 $, and $x$ on the boundary $\pd \O$, there exists a domain $\tilde \O$ whose symmetric difference with $ \O $ is contained in the ball of radius $\e$ centered at $x$, and whose (Dirichlet, Neumann, or Robin) Laplacian eigenvalues are all simple.
Moreover $\tilde \O$ can be constructed so that the Lipschitz constant of $\pd \tilde \O$ is arbitrarily near to the one of $\pd \O$.
\end{thm}

More in detail the structure of the paper is the following.
In \Cref{sect:pre} we review some preliminary material, in particular regarding spectral stability.
In \Cref{sect:Hadamard} we recall a Hadamard's formula and study some independence properties of eigenfunctions and their gradients at the boundary.
More in detail, Hadamard's formula provides us with a first-order estimate on the shift of an eigenvalue $\l$ which depends on the value of
\begin{equation}
\label{eq:u-lu}
|\nabla u|^2 - c u^2
\end{equation}
at the boundary of the domain considered, where $u$ is an eigenfunction associated to $\l$ and $c$ is a constant which depends only on the choice of boundary conditions.
By showing that for two orthogonal eigenfunctions the corresponding values of \eqref{eq:u-lu} in any open subset of the boundary must differ at least at a point, we are able to construct a localized perturbation which splits any non-simple eigenvalue.
However, even when small, this perturbation might cause the shift and the overlap of other eigenvalues. 
This possibility is ruled out in \Cref{sect:proofthm}, where uniform bounds for the whole spectrum are adapted to our case from sharp stability estimates from \cite{BurenkovLamberti2}.
In conclusion, these bounds allow the construction of a localized perturbation, which consists of a sequence of small ``bumps" at the boundary of the domain considered, which proves \Cref{T:EllipticLocalSpectrumGen}.

\section{Notations and preliminary results}
\label{sect:pre}

In this section we fix the main notation which will be used in the paper and recall some preliminary results on eigenvalues and eigenfunctions of the Laplacian.
Regarding the notation:
\begin{itemize}
\item we say that $X$ is a domain if $X$ is an open, bounded, and connected subset of $\RR^N$;
\item we say that $\l$ is an eigenvalue of a domain $X$ with associated eigenfunction $u$ (assumed to be not constant zero) if 
\begin{equation}
\label{eq:LaplaceEigen}
\Delta u + \l u = 0  \qquad \text{in } X,
\end{equation}

and either one of the following homogeneous boundary conditions is satisfied on $\pd X$:
\begin{equation}
\label{eq:bdyCond}
\begin{cases}
u = 0 & \text{(Dirichlet)}, \\[2pt]
\dfrac{\pd u}{\pd \nu} = 0 & \text{(Neumann)}, \\[4pt]
\sigma u = \dfrac{\pd u}{\pd \nu} & \text{(Robin)},
\end{cases}
\end{equation}
where $\sigma $ is a fixed non-zero constant and $\nu$ indicates the outward unit normal vector.
\item we indicate as $\O$ a fixed domain with Lipschitz boundary.
\end{itemize}

We actually require  \eqref{eq:LaplaceEigen} and \eqref{eq:bdyCond} to be satisfied only in a weak sense, that is: $\l $ is an eigenvalue of $X$ with associated eigenfunction $u$, if $u$ is an element of a function space $V(X)$ and
$$Q(u,v) = \l \int_X u v , \qquad \text{ for every } v \in V(X),$$
where, depending on the choice of boundary conditions, we have
\begin{equation}
\label{eq:weakbdCond}
\begin{tabular}{c|c|c}
Boundary conditions &  $Q(u,v)$ & $V(X)$\\[2pt]
\hline & & \\[-9pt]
Dirichlet & $\int_X \nabla u \cdot  \nabla v $ & $\bra{ u \in H^1(X) : \text{trace of } u \text{ at } \pd X \text{ is } 0}  $\\[3pt]
Neumann & $\int_X \nabla u \cdot \nabla v $ & $ H^1(X)  $\\[3pt]
Robin & $\int_X \nabla u \cdot \nabla v - \int_{\pd X} \sigma u v$ & $H^1(X)  $\\
\end{tabular}
\end{equation}
where $H^1$ is the space of square integrable functions with square integrable distributional gradient.
However, from elliptic regularity theory, we know that Laplacian eigenfunctions are analytic inside any open domain.
Thus \eqref{eq:LaplaceEigen} is satisfied also in the classical sense.
Moreover if $\Sigma$ is a smooth (that is $C^\infty$) part of $\pd X$, $u$ is also smooth on $\Sigma$ (see for example \cite[Section 6.3]{Evans}).

Recall from spectral theory that the eigenvalues of $\O$ have finite multiplicity and can be arranged in a non-decreasing sequence which tends to infinity, and which we will denote as
$$\l_1 \leq \l_2 \leq \dots ,$$
where each eigenvalue is repeated as many times as its multiplicity.

For future reference we record the following uniqueness result.

\begin{thm}
\label{T:HolmgrenUniq}
Let $u$ be such that $\Delta u + \l u = 0$ in $\O$.
If $u= 0$ and $\dfrac{\pd u}{\pd \nu} = 0$ on $\Sigma$, an open and smooth subset of $\pd \O$, then $u$ is constant zero in the whole $\O$.
\end{thm}

We briefly outline the classic argument to prove this fact from Holmgren's uniqueness theorem.
Let $B$ be an open ball such that $B \cap \pd \O \sset \Sigma$.
Extending $u$ to $0$ in $B \smin \O$, it is easy to check that $-\Delta u = \l u$ in the distributional sense in $B$.
By \cite[Theorem 5.3.1]{Hormander}, $u$ must be zero also in an open set inside $\O$.
But then $u = 0$ on the whole $\O$ by analytic continuation.

\subsection{Stability of eigenvalues of the Laplacian}

We review some results that show that the spectrum of the Laplacian is continuous under domain perturbations, and give some useful quantitative estimates on the eigenvalues' shifts.

First we recall a result of analyticity of eigenvalues and eigenfunctions with respect to a perturbation parameter, which is a consequence of the classic Rellich-Nagy Theorem \cite[Theorem 1 at p.~33]{Rellich} (see also \cite[Section 4.2]{BurenkovLambertiCristoforis} and references therein).

%We recall some results from \cite{LemenantMilakisSpinolo} which provide us with estimates on the stability of eigenvalues with respect to the Hausdorff distance between the underlying domains.

%We consider now the eigenvalue problems \eqref{eq:LaplaceEigen}, with boundary condition \eqref{eq:bdyCond}, on $D$ and $\O_\e$.
%Let $V$ be a dense subset of $H^1$ which takes into account the boundary conditions on the boundary.
%For example, for homogeneous Dirichlet boundary conditions $V = H_0^1$ (the closure in the $H^1$ norm of $C^\infty$ functions compactly supported in $D$), while for homogeneous Neumann boundary conditions $V=  \bra{u \in H^1: \int u = 0}$.

\begin{thm}
\label{T:RellichNagy}
Let $(\phi_t)_{t \in [0,t_0]}$ be a family of diffeomorphisms of $\RR^N$ such that $\phi_t$ is analytic in $t$, $\phi_0 $ is the identity, and $\phi_t (\O) \supseteq \O $ for every $t$.
Let $\l$ be an eigenvalue of $\O$ of multiplicity $m$.
Then there exist $\l_{t}^1 \leq \dots \leq \l_{t}^m$ and functions $u_{t}^1, \dots, u_{t}^m$ such that for $j = 1, \dots, m$,
\begin{itemize}
\item for any $t$, $\l_t^j$ is an eigenvalue of $\O_t$ with associated eigenfunction $u_t^j$;
\item for any $t$, $\int_{\O_t} u_t^j u_t^i $ is $1$ if $j=i$ and is $0$ otherwise;
\item $\l_t^j $ and $u_t^j$ are analytic in $t$;
\item $\l_0^j =\l$ and $u_0^j$ is an eigenfunction associated to $\l$.
\end{itemize}
Moreover for any $\delta > 0$ small enough, there is a $T$ such that for any $t < T$ the only eigenvalues of $\phi_t(\O)$ in $(\l - \delta, \l+\delta)$ are $\l_t^1, \dots ,\l_t^m$.
\end{thm}

For our purposes we will also need a finer estimate on the variation of eigenvalues, as expressed in the following lemma.

\begin{lemma}
\label{T:QuantEstEigVar}
Let $\phi $ be a diffeomorphism of $\RR^N$.
Let $\l_n$ be the $n$-th eigenvalue of $\O$ and $\tilde \l_n$ the $n$-th eigenvalue of $\phi( \O)$.
Then there exists a constant $C$, which depends only on the Lipschitz constants of $\pd \O$ and of $\phi$, such that
$$|\tilde \l_n - \l_n| \leq C \max \{\tilde \l_n, \l_n \} (|\phi - id|_{C^{1}(\overline{\O})}). $$
\end{lemma}

The proof of this estimate can be obtained by following the same argument in the proof of \cite[Lemma 6.1]{BurenkovLamberti2}, substituting appropriately the bilinear form and the function space with the ones defined in \eqref{eq:weakbdCond}, depending on the boundary conditions considered.

\section{Hadamard's formula and boundary properties of eigenfunctions}
\label{sect:Hadamard}

In this section we study some independence properties of Laplacian eigenfunctions and of their gradients at the boundary.
We first recall a Hadamard's formula for the variation of eigenvalues under a deformation of the boundary.
The dot superscript will indicate differentiation in $t$.

\begin{lemma}
\label{T:Hadamard}
Let $(\phi_t)_{t\in [0,t_0]}$  be a family of diffeomorphisms
such that $\phi_t $ is analytic in $t$ and $\phi_0$ is the identity.
%Let $(\O_t)_{t\in [0,t_0]}$ be a family of domains such that $\O_t = \phi_t (\O)$ for every $t$.
Suppose that the support of $\phi_t$ is contained in a fixed open set $U$ for every $t$, and that $\pd \O \cap U$ is smooth.
Let $\l_t, u_t$ be an eigenvalue-eigenfunction couple of $\phi_t(\O)$, and suppose both are differentiable in $t$.
Then
\begin{equation}
\label{eq:Hadamard}
\dot \l_0 = \int_{\pd \O} \Big(|\nabla u_0|^2 - \l_0 u_0^2 + (\pd_{\nu_0} u_0) (Hu_0 - 2 \pd_{\nu_0} u_0) \Big) \nu_0 \cdot \dot e_0,
\end{equation}
where $\nu_t$ indicates the outward unit normal vector, $e_t$ the identity on $\phi_t(\pd \O)$, and $H$ is the mean curvature of $\pd \O$.
\end{lemma}

Hereafter we briefly prove this fact in the case of homogeneous Dirichlet or Neumann boundary conditions.
The case of Robin conditions requires a finer analysis of the dependence on $t$ of the surfaces $\phi_t(\pd \O)$, for which we refer to \cite[Identities (69) and (57)]{BandleWagner}.

\begin{proof}
Let $(\O_t)_{t\in [0,t_0]}$ be a family of domains such that $\O_t = \phi_t (\O)$ for every $t$.
By the divergence theorem, the distributional gradient of the measure $\1_{\O_t} \d \LL^N$, where $\1_{\O_t}$ is the characteristic function of $\O_t$ and $\LL^N$ is the $N$-dimensional Lebesgue measure, is given by $\nu_t \Sigma^{N-1}_t $, where $\Sigma^{N-1}_t$ is the surface measure on $\pd \O_t$.
Therefore by the chain rule
$$\dfrac{d}{dt} ({\1}_{\O_t} \, \LL^N)   =  \nu_t \cdot \dot e_t \, \Sigma^{N-1}_t,$$
so we have the following Leibniz' formula:
\begin{equation}
\label{eq:LeibnizSurfaces}
\dfrac{d}{dt} \left( \int_{\O_t} f_t  \right) = \int_{\O_t} \dot f_t  + \int_{\pd \O_t} f_t \nu_t \cdot \dot e_t.
\end{equation}
Consider now the identity
\begin{equation}
\label{proof:eq:Hadamard:id}
\l_t = - \int_{\O_t} u_t \Delta u_t = \int_{\O_t}|\nabla u_t|^2 .
\end{equation}
Differentiating in $t$ the first equality in \eqref{proof:eq:Hadamard:id} and using \eqref{eq:LeibnizSurfaces} we obtain
\begin{equation}
\label{proof:eq:Hadamard:firstdid}
2 \l_t \int_{\O_t} \dot u_t u_t = -\l_t \int_{\pd \O_t} u_t^2 \nu_t \cdot \dot e_t.
\end{equation}
In the case of Neumann boundary conditions,
differentiating in $t$ the last term in \eqref{proof:eq:Hadamard:id}, using \eqref{eq:LeibnizSurfaces}, integrating by parts, and substituting \eqref{proof:eq:Hadamard:firstdid}, we have that
\begin{equation*}
\dot \l_t = \int_{\pd \O_t} (|\nabla u_t|^2 - \l_t u_t^2) \nu_t \cdot \dot e_t + 2 \int_{\pd \O_t} \dot u_t \dfrac{\pd u_t}{\pd \nu_t},
\end{equation*}
which gives \eqref{eq:Hadamard} since $\pd_{\nu_0} u_0 = 0$ on $\pd \O_0$.
Proceeding in the same way for Dirichlet boundary conditions, only exchanging the roles of the functions in the integration by parts step, we obtain
\begin{equation*}
\dot \l_t = \int_{\pd \O_t} (|\nabla u_t|^2 - \l_t u_t^2) \nu_t \cdot \dot e_t + 2 \int_{\pd \O_t} u_t \dfrac{\pd {\dot u_t}}{\pd \nu_t} + 2 \dot \l_t,
\end{equation*}
which gives \eqref{eq:Hadamard} since $u_0 = 0$ on $\pd \O_0$.
\end{proof}
We notice that considering
\begin{equation}
\label{eq:cBdyCond}
c =
\begin{cases}
0 & \text{ if } u|_{\pd \O} = 0 ,\\[2pt]
\l_0 & \text{ if } \pd_\nu u |_{\pd \O}= 0 ,\\[2pt]
\l_0 + 2 \sigma^2 & \text{ if } \sigma u |_{\pd \O} = \pd_\nu u |_{\pd \O},\\  
\end{cases}
\end{equation}
if $\dot e_0$ is supported on a flat part of $\pd \O$, the integrand in \eqref{eq:Hadamard} can be rewritten as $|\nabla u|^2 - c u^2$.
In the following lemma we study such a quantity, in particular the behavior of its zeros.

\begin{lemma}
\label{T:Gradu-lunotconst}
Let $c$ be a constant and let $u, \tilde u$ be two orthonormal eigenfunctions associated to the same eigenvalue.
Let $\Sigma $ be an arbitrary smooth open subset of $\pd \O$.
Then:
\begin{enumerate}
\item $|\nabla u|^2 - c u^2$ cannot be constant zero on $\Sigma$;
\item \label{T:Gradu-lunotconst:2}
 $|\nabla u|^2 - c u^2 - (|\nabla \tilde u|^2 - c \tilde u^2)$ cannot be constant zero on $\Sigma$.
\end{enumerate}
\end{lemma}

\begin{proof}
The thesis for the case $c = 0$ is given by \Cref{T:HolmgrenUniq}.
Consider $c \neq 0$.
Our approach is inspired to the treatment of \cite[Chapter 6]{Henry}.

We first prove Point 1.
Suppose by contradiction that $ |\nabla u|^2 = c u^2$ on $\Sigma $.
We consider separately the different possible boundary conditions in \eqref{eq:bdyCond}.
\begin{enumerate}[i)]
\item \label{proof:Gradu-lunotconst:caseB=0} 
If the Dirichlet condition holds then $\pd{u}/\pd{\nu} = u = 0$ on $\Sigma$.
By \Cref{T:HolmgrenUniq} then $u=0$ on $\O$, a contradiction.
\item \label{proof:Gradu-lunotconst:caseA=0}
Suppose the Neumann condition holds.
The eigenfunction $u$ cannot be constant $0$ on $\Sigma$, otherwise we would be again in the situation of Case \ref{proof:Gradu-lunotconst:caseB=0}, so there is $x_0 \in \Sigma $ s.t.  $u(x_0) \neq 0$.
Let $\g_t$ be a solution in $\Sigma$ of the ODE
$$\begin{cases}
\g_0 = x_0, \\
\dot \g_t = C \nabla u(\g_t),
\end{cases}$$
 with $C $ a constant to be determined.
Then
\begin{equation}
\label{proof:Gradu-lunotconst:blowingcurve}
\dfrac{du(\g_t) }{dt} = C|\nabla u(\g_t)|^2 = C c u(\g_t)^2,
\end{equation}
if $\g_t \in \Sigma$.
Therefore by choosing $C$ large enough, there will be a time $T$ at which $\g_T \in \Sigma$ and $|u(\g_t) |\xrarr[t \to T]{}  \infty$, which is a contradiction.
\item \label{proof:Gradu-lunotconst:caseAandBnot0}
If the Robin condition holds, then 
$$c u^2 = |\nabla u|^2 = \sigma^2 u^2 + |\nabla_{\! \! S} \, u|^2 \quad \text{ on } \Sigma,$$
where $\nabla_{\! \! S} \, u$ is the surface gradient of $u$ on $\pd \O$.
If $c \neq \sigma^2$, we can build, as in Case \ref{proof:Gradu-lunotconst:caseA=0}, a curve $\g$ on which the eigenfunction $u$ blows up in short time, leading to a contradiction.
If $c = \sigma^2$ then $|\nabla_{\! \! S} \, u| = 0$ on $\Sigma$, and this leads to the following chain of implications: $u$ is constant on $\Sigma$, $\pd u / \pd \nu$ is constant on $\Sigma$, $u$ is  constant in $\O$ by \Cref{T:HolmgrenUniq}, $\pd u / \pd \nu$ is zero on $\pd \O$, $u$ is zero on $\O$ by \Cref{T:HolmgrenUniq}, a contradiction.
\end{enumerate}

We now prove Point 2.
Suppose by contradiction that $|\nabla u|^2 - |\nabla \tilde u|^2 = c ( u^2 - \tilde u^2 )$ on $\Sigma$.
Let $x_0 \in \Sigma$ be a point where $u(x_0)$ and $\tilde u(x_0)$ are different (existence of such a point is guaranteed by the smoothness of eigenfunctions on $\Sigma$ and \Cref{T:HolmgrenUniq}).
Let $f_t = u(\g_{t}), \tilde{f}_t = -\tilde u(\tilde \g_{t})$, where $\g$ and $\tilde \g$ solve
$$\begin{cases}
\dot \g_t = C \nabla u(\g_t), \\
\dot{\tilde{\g}}_t = -C\nabla \tilde u (\g_t),\\
\g_0 = \tilde \g_0 =x_0,
\end{cases}$$ and $C$ is a constant to be determined.
Then 
$$\dot f_t+\dot{\tilde{f}}_t =  C c ( f_t^2 + \tilde{f}_t^2) .$$
Therefore $\dot f_t \geq C c f_t^2$ or $\dot{\tilde{f}}_t \geq C c \tilde{f}_t^2$ for $t$ in a small neighborhood of $0$.
In conclusion, a choice of $C$ large enough would lead to blow up in short time of $u$ or $\tilde u$, which is impossible.
\end{proof}

\section{Splitting of the spectrum}
\label{sect:proofthm}

With the tools developed so far we can construct a localized boundary deformation which splits the eigenvalues perturbed from one eigenvalue as follows.

\begin{proposition}
\label{T:OneSplit}
Let $x \in \pd \O$, $B$ a ball centered at $x$, and $\Sigma = B \cap \pd \O$.
Suppose $\Sigma$ is flat, that is $\Sigma $ is contained in a hyperplane. 
Then, under the same hypotheses and notation of \Cref{T:RellichNagy}, we can construct a family of diffeomorphisms $(\phi_t)_{t \in (0,t_0)}$ such that $\phi_t$ is the identity outside $B$, $|\phi_t - id|_{C^1}$ is arbitrarily small, and $\l_t^i \neq \l_t^j$ for any $i,j \in \bra{1, \dots, m}$ and for all $t \in (0,t_0)$.
\end{proposition}

\begin{proof}
Let $c$ be as in \eqref{eq:cBdyCond}.
By Point \ref{T:Gradu-lunotconst:2} of \Cref{T:Gradu-lunotconst}, there exists $y $ on $\Sigma$ such that 
\begin{equation}
\label{proof:eq:Single:Ineq}
(|\nabla u^i_0|^2 - c (u^i_0)^2)(y) \neq (|\nabla u^j_0|^2 - c (u^j_0)^2)(y).
\end{equation}
Then, by choosing a deformation of the boundary $\phi_t$ which is the identity outside an appropriately small neighborhood of $y$, we have
\begin{equation}
\label{proof:eq:Single:Had}
\int_{\pd \O} (|\nabla u^i_0|^2 - c (u^i_0)^2) \nu \cdot \dot \phi_0 \neq  \int_{\pd \O} (|\nabla u^j_0|^2 - c (u^j_0)^2) \nu \cdot \dot \phi_0.
\end{equation}
Such a perturbation can be constructed in many ways; for the sake of completeness, we give an explicit example hereafter.
 
By eventually reducing to a smaller $B$ and applying an invertible affine transformation, we can assume that $y = 0$ and $\Sigma = \bra{z \in B_1 : z_N = 0}$, where $B_1$ is the unit ball. 
Let $\hat z $ indicate $ (z_1, \dots, z_{N-1})$ and let
$$\rho_c(\hat z) = \begin{cases}
c^2 \exp \left( \dfrac{1}{|\hat z/c|^2-1} \right) & \text{ if } |\hat z| < c,\\[4pt]
0 & \text{ otherwise}.
\end{cases}$$
Notice that by construction $|\rho_c|_{C^1} \leq c$ for any $c \leq 1$. 
Let $\phi_t(z)$ be the extension of the map $ z \mapsto (\hat z , t \rho_c(\hat z))$ from  $\Sigma$ to a smooth function which is the identity outside $B$ and such that $|\phi_t - id|_{C^1} \leq |\rho_c|_{C^1}$.
By construction, $\nu \cdot \dot \phi_0 = \rho_c(\hat z)$ on $\Sigma$.
Then by choosing $c$ small enough, by the smoothness of $u$ on $\Sigma$ and by \eqref{proof:eq:Single:Ineq}, we have that \eqref{proof:eq:Single:Had} holds. Moreover we remark that it holds
\begin{equation}
\label{proof:eq:boundc}
|\phi_t - id|_{C^1} \leq c.
\end{equation}

In conclusion, by \Cref{T:Hadamard}, \eqref{proof:eq:Single:Had} implies that $\dot \l^i_0 \neq \dot \l^j_0 $.
Since $\l^i_t$ and $\l^j_t$ are both analytic in $t$, there exists a small $t_0$ such that $\l_t^i \neq \l_t^j$ for $t \in (0,t_0)$.
\end{proof}

\begin{rmk}
\label{T:LocPertFlat}
The flatness assumption of $\Sigma$, although making the argument simpler, is not really necessary in the proof of \Cref{T:OneSplit}, as one might build a boundary deformation such that \eqref{proof:eq:Single:Had} holds even if $\Sigma$ is not flat;
the idea would be the same, only some care would be required to manage the mean curvature term which is present in  \eqref{eq:Hadamard}.
On the other hand, if our aim is to find a local perturbation as in \Cref{T:EllipticLocalSpectrumGen}, the flatness assumption is not restrictive.
In fact, if $\Sigma $ is not contained in a hyperplane, by eventually considering a smaller $B$ and changing basis, we can assume that $\Sigma$ is the graph of a Lipschitz function $\phi$ such that $\phi(0) = x = 0$. 
Let  $B_r, B_R$ be two balls centered in $0$ such that $B_r \subset B_R \subset B $, and let $\eta$ be a smooth function which is $0$ in $B_r$ and $1$ outside $B_R$. 
Then the graph of $\phi \eta$ will be flat in $B_r$.
Notice also that as $r \to 0$, $\eta$ can be chosen so that the Lipschitz constant of $\phi \eta$ converges to the Lipschitz constant of $\phi$.
Thus for any $\delta > 0$, we can build a Lipschitz domain which differs from $\O$ only in $B$, is flat in $B_r$ (for a certain $r$ which depends on $\delta$), and whose Lipschitz constant differs from the Lipschitz constant of $\O$ by less than $\delta$.
\end{rmk}

We further remark that although \Cref{T:OneSplit} shows how to split one  eigenvalue, the perturbation chosen might cause a couple of two other eigenvalues to overlap, creating a new repeated eigenvalue.
To avoid this problem we need a finer control on the behavior of the whole spectrum; this is what is achieved in the following lemma.

\begin{lemma}
\label{T:SingleSplit}
Consider $\e >0$, $x$ a point on the boundary $ \pd \O$, and $\l_r$ the first eigenvalue of $\O$ of multiplicity $m \geq 2$.
Then for any $M > 0$ there exists a Lipschitz domain $\tilde \O$, whose eigenvalues we indicate as $\tilde \l_1 \leq \tilde \l_2 \leq \dots$, such that:
\begin{enumerate}
\item \label{T:SingleSplit:1}
the symmetric difference $\tilde \O \vartriangle \O$ is contained in the ball of radius $\e$ centered at $x$;
\item \label{T:SingleSplit:2}
for all $i \leq r+m+1$, it holds $|\tilde \l_i - \l_i| \leq M d_r$, where $d_r$ is the minimum positive number of the set $\bra{\l_{j+1} - \l_{j} : j=1, \dots, r+m} $;
\item \label{T:SingleSplit:3}
the multiplicity of $\tilde \l_r$ is strictly smaller than the multiplicity of $\l_r$;
\item \label{T:SingleSplit:4}
for all $i > r+m$, it holds $ \tilde \l_i > \l_{r} $.
\end{enumerate}
\end{lemma}

\begin{proof}
Let $B_\e$ be the ball of radius $\e$ centered at $x$ and let $\Sigma = B_\e \cap \pd \O$.
With the same construction of \Cref{T:LocPertFlat} and of the proof of \Cref{T:OneSplit}, we can build $(\O_t)_{t \in (0,t_0)}$ a family of perturbations of $\O$ obtained by a deformation of the boundary of $\O$  localized in $B_\e$.
Let $\l_1^t, \l_2^t, \dots$ indicate the sequence of eigenvalues of $\O_t$, with associated eigenfunctions $u_1^t, u_2^t, \dots$.
By \Cref{T:RellichNagy} we can assume that $\l_i^t, u_i^t$ are analytic in $t$, that $\l_i^0 = \l_i$, and that $u_r^0 ,\dots , u_{r+m}^0$ is an orthonormal basis for the eigenspace of $\l_r$.
By \Cref{T:OneSplit}, there are two distinct indices $i$ and $j$ among $\{ r, \dots, r+m \}$,  such that for $t_0$ small enough
\begin{equation}
\label{proof:eq:SingleSplit:2}
\l_i^t \neq \l_j^t, \quad \text{ for } t \in (0, t_0).
\end{equation}
By the eigenvalue stability estimate of \Cref{T:QuantEstEigVar}, there is a $t_0$ small enough such that 
\begin{equation}
\label{proof:eq:SingleSplit:3}
| \l_i^t - \l_i| \leq M d_r, \quad \forall t < t_0, \forall i \in \bra{1 , \dots, r+m+1}.
\end{equation}
Let $C, C'$ indicate two constants which depend only on the dimension $N$, the Lipschitz constant of $\pd \O$  and the area of $\O$. 
By Weyl's asymptotic law, $\l_n = C n^{2/N} + o(n^{2/N})$ for any $n$.
Then, from the uniform estimate of \Cref{T:QuantEstEigVar}, for $i > r+m$ it holds
$$ \l_i^t - \l_r \geq (\l_i^t - \l_i) + \l_i - \l_r \geq C'(-C c i^{2/N} +i^{2/N} - r^{2/N}),$$
where $c > 0$ is a bound on the deformation magnitude (which we can choose arbitrarily small) as in \eqref{proof:eq:boundc}.
Therefore for $t_0$ and $c$ small enough,
\begin{equation}
\label{proof:eq:SingleSplit:4}
\l_i^t - \l_r > 0, \quad \forall t < t_0, \forall i > r+m.
\end{equation}
In conclusion, taking $\tilde \O \deq \O_t$ for a certain $t$ small enough, Point \ref{T:SingleSplit:1} of the thesis holds by construction while Points \ref{T:SingleSplit:2}-\ref{T:SingleSplit:3}-\ref{T:SingleSplit:4} are consequences of \eqref{proof:eq:SingleSplit:3}-\eqref{proof:eq:SingleSplit:2}-\eqref{proof:eq:SingleSplit:4}.
\end{proof}

The construction in the previous proof gives us a method to split the first non-simple eigenvalue without altering the simplicity of smaller eigenvalues.
In fact by taking $M < 1/2$, from Points \ref{T:SingleSplit:2} and \ref{T:SingleSplit:4} of \Cref{T:SingleSplit} we have that the eigenvalues $\tilde \l_i$ perturbed from  $\l_i$:
\begin{itemize}
\item lie in disjoint neighborhoods of $\l_i$, for $i < r$;
\item  are not further than $d_r/2$ from $\l_i$, for $r \leq i \leq r+m $;
\item  are larger than $\l_r$, for $i > r+m$.
\end{itemize}
Therefore $\tilde \l_1, \dots, \tilde \l_{r-1}$ must still be simple.
We can iterate this procedure to split the whole spectrum as in the following proof.

\begin{proof}[Proof of \Cref{T:EllipticLocalSpectrumGen}]
Denoting as $B_\e$ the ball of radius $\e$ centered at $x$, let $\Sigma = B_\e \cap \pd \O$.
As in \Cref{T:LocPertFlat}, for any $\delta > 0$, we can modify $\Sigma$ into $\Sigma'$ so that an open subset of $\Sigma'$ is contained in a hyperplane and the Lipschitz constant of $\Sigma'$ differs from the Lipschitz constant of $\Sigma$ by less than $\delta$.
Let $(B_n)_{n \in \NN}$ be a sequence of disjoint balls of radius $c 2^{-n}$ with centers on $\Sigma'$ and contained in $B_\e$, with $c$ small enough so that $\Sigma' \cap \bigcup_n B_n$ is flat.
In each $B_n$ we deform $\Sigma'$ with a diffeomorphism $\phi_n$ built as in the proof of \Cref{T:OneSplit}.
We obtain this way a sequence of domains $(\O_n)_{n \in \NN}$ such that the thesis of \Cref{T:SingleSplit} holds with $\O, \tilde \O, B, M$ replaced respectively by $\O_n, \O_{n+1}, B_n, M_n$ for each $n$, where for $M_n$ we take a constant smaller than $1/2^{n+1}$.
Additionally, we can take $\phi_n$ such that $|\phi_n - id|_{C^1} \leq \delta/n$. And thus as $n \to \infty$, $\O_n $ converges to a domain $\tilde \O $ with Lipschitz constant not farther than $\delta$ from the Lipschitz constant of $\O$.

Let $r_n$ be the index of the first non-simple eigenvalue of $\O_n$.
By Points \ref{T:SingleSplit:2} and \ref{T:SingleSplit:4} of \Cref{T:SingleSplit} we have that all eigenvalues with index smaller than $r_n$ are simple for any $n$.
Moreover $r_n $ is a non-decreasing sequence of integers which cannot be definitely constant; in fact by Point \ref{T:SingleSplit:3} of \Cref{T:SingleSplit}, $r_{n+j}$ can be equal to $r_n$ for at most $j \in \bra{1, \dots, r_n}$.
Therefore $r_n \to \infty$ as $n \to \infty$, and thus $\tilde \O$ can have only simple eigenvalues.
\end{proof}

%As a final remark, we mention that we are faithful that a result analogue to \Cref{T:EllipticLocalSpectrumGen} also holds for more general elliptic operators and does not depend on the particular choice of boundary conditions.

\bibliographystyle{abbrvurl}
\bibliography{ref}

\end{document}